\documentclass[abstract,twocolumn,fontsize=10pt,DIV=27]{scrartcl}

\usepackage{amsmath}
\usepackage{amssymb}
\usepackage{amsthm}
\usepackage{enumitem}
\usepackage[colorlinks = true, linkcolor = blue, citecolor = blue]{hyperref}
\usepackage[nameinlink]{cleveref}
\usepackage{graphicx}
\usepackage{natbib}
\usepackage{wrapfig}

\setcapindent{0pt}
\setkomafont{caption}{\footnotesize}

\newcaptionname{USenglish}{\figurename}{Fig.}

\theoremstyle{definition}
\newtheorem{assumption}{Assumption}
\newtheorem{definition}{Definition}
\newtheorem{example}{Example}
\newtheorem{proposition}{Proposition}
\newtheorem{remark}{Remark}
\newtheorem{theorem}{Theorem}

\crefname{equation}{}{}
\crefname{assumption}{Assumption}{Assumptions}
\crefname{proposition}{Proposition}{Propositions}

\title{Extremum Seeking with Double Integrators \\ in the Presence of Local Extrema}
\author{Raik Suttner, \ Christian Ebenbauer, \ and \ Sergey Dashkovskiy}
\date{}

\begin{document}
\maketitle
\begin{abstract}
We study the problem of global extremum seeking in the presence of local extrema. We investigate two different perturbation-based methods: $\vphantom{(}$1) a well-known classical extremum seeking scheme for steady-state output optimization, and $\vphantom{(}$2) a source seeking scheme for a two-dimensional point mass. In each of these two scenarios, the closed-loop system involves a damped double integrator subject to an oscillatory force. An averaging analysis reveals that the respective averaged system is again a damped double integrator, but now subject to a potential force. The potential force is given by the gradient of a locally averaged objective function. Such a function is less prone to have undesired local extrema and is therefore better suited for global optimization. We provide sufficient conditions for semi-global practical uniform asymptotic stability of the closed-loop systems. The sufficient conditions only involve assumptions on the averaged objective function but not the original one.
\end{abstract}

\section{Introduction}\label{sec:1}
The intention of extremum seeking control is to steer a system towards a state where an output function attains an extreme value. Many different methods have been proposed \cite{Scheinker2024}; such as biology-inspired stochastic algorithms \cite{LiuBook}, differential geometric strategies for systems on manifolds \cite{Duerr2014}, or delay- and PDE-based schemes \cite{OliveiraBook}. Global stability results for extremum seeking systems are typically proved under the assumption that the output (or objective) function has a unique global extremum but no other critical point \cite{Tan2006}.

Global extremum seeking in the presence of local extrema is a very challenging problem. In simulations, it is often observed that the choice of the employed dither signals is decisive for overcoming local extrema \cite{Nesic2006}, \cite{Wang2016}, \cite{Bhattacharjee2021}. A global stability result for functions with local extrema is presented in \cite{Tan2009}, which is based on a bifurcation analysis of equilibria for different dither amplitudes. The dither amplitudes in \cite{Tan2009} are initially large and then decrease slowly to zero. Extensions to schemes with adaptive dither amplitudes can be found, e.g., in \cite{Ye2020}, \cite{Mimmo2023}. A sampled-data approach without dither signals is presented in \cite{Khong20132}, where a known numerical optimization method is used for finding a global extremum within a user-prescribed compact set. Another perspective on global extremum seeking, which is very similar to the one in the present paper, can be found in \cite{Wildhagen2018}: It is conjectured that extremum seeking control optimizes not the original objective function but a different one.

In the recent studies \cite{Suttner20241}, \cite{Suttner20243}, and \cite{Suttner20261}, suitably designed extremum seeking schemes provide access to the gradient of a locally averaged objective function. To this end, the values of the objective function are measured and integrated along a periodic trajectory in the state space. In \cite{Suttner20241} and \cite{Suttner20243}, this approach is investigated for a source-seeking unicycle, where the circular motion of the sensor (in dimension $2$) allows integration of the signal along circles. In \cite{Suttner20261}, a nearly spherical motion of a single integrator point in dimension $\geq2$ allows integration of the objective function over spheres. Using the divergence theorem, one can show that the integral of the objective function over a circle (resp.~sphere) is equal to the gradient of a locally averaged objective function, where the local average is taken over the enclosed disk (resp.~ball). The state of the closed-loop system is driven into the gradient direction of the averaged objective function. This property can be very beneficial for the purpose of global optimization. A local average can ``wash out'' undesired local extrema in which a standard gradient-based method would get stuck.

In the present paper, we extended the above concept of a locally averaged objective function to two further classes of extremum seeking systems. First, we study a known classical extremum seeking scheme for steady-state output optimization. The filter dynamics of this scheme are interpreted as the motion of a double integrator point mass under periodic forcing. For a suitable choice of the control parameters, we show that the double integrator is driven into the gradient direction of a locally averaged steady-state output function. Secondly, we study the problem of source seeking with a two-dimensional double integrator point mass. Again, for a suitable choice of the control parameters, we show that the double integrator is driven into the gradient direction of a locally averaged signal function. For each of the two types of systems, we provide semi-global stability results. The results only involve assumptions on the averaged objective function and do not necessarily exclude local extrema of the actual objective function. We provide explicit examples of objective functions with local extrema to which the semi-global stability results can be applied. The proposed methods are also tested in numerical simulations.

The paper offers a novel point of view on the behavior of two extremum seeking schemes in the presence of local extrema. It provides a precise mathematical explanation why these two schemes have the ability to overcome local extrema if the dither amplitudes are chosen in a suitable way. In contrast to previous work, which focused on first-order kinematic gradient systems, in the present paper, we investigate second-order dynamic gradient systems of locally averaged objective functions. This way, we combine two potentially beneficial features for overcoming local extrema: local averaging and momentum.

%
%
%
%

\section{Preliminaries}\label{sec:02}
As usual, let $\mathcal{KL}$ denote the set of all continuous functions $\beta\colon\mathbb{R}_{\geq0}\times\mathbb{R}_{\geq0}\to\mathbb{R}_{\geq0}$ such that $\vphantom{(}$1) for each $t\in\mathbb{R}_{\geq0}$: $\beta(0,t)=0$ and $\beta(\cdot,t)$ is strictly increasing, as well as $\vphantom{(}$2) for each $r\in\mathbb{R}_{\geq0}$: $\beta(r,\cdot)$ is non-increasing and $\beta(r,t)\to0$ as $t\to\infty$. In the subsequent definition, we consider a time- and parameter-dependent system of the form%
\begin{equation}\label{eq:01}
\dot{x} \ = \ g(x,t,\epsilon,\delta),
\end{equation}%
where $g\colon\mathbb{R}^n\times\mathbb{R}\times\mathbb{R}_{>0}\times\mathbb{R}_{>0}\to\mathbb{R}^n$ is smooth, $x\in\mathbb{R}^n$ is the state, $t\in\mathbb{R}$ is the time, and $\epsilon,\delta\in\mathbb{R}_{>0}$ are parameter.%
\begin{definition}\label{definition:01}
Let $x^\ast\in\mathbb{R}^n$. The time- and parameter-dependent system \cref{eq:01} is said to be \emph{semi-globally practically uniformly asymptotically stable} (SGPUAS) w.r.t.~$x^\ast$ if there exists $\beta\in\mathcal{KL}$ such that, for all $r,\sigma>0$, there exists $\delta_0>0$ such that, for every $\delta\in(0,\delta_0)$, there exists $\epsilon_0>0$ such that, for every $\epsilon\in(0,\epsilon_0)$, every $x_0\in\mathbb{R}^n$ with $|x_0-x_\ast|\leq{r}$, and every $t_0\in\mathbb{R}$, the solution $x$ of \cref{eq:01} with initial condition $x(t_0)=x_0$ satisfies%
\begin{equation}\label{eq:02}
|x(t)-x^\ast| \ \leq \ \beta(|x_0-x^\ast|, t-t_0) + \sigma
\end{equation}%
for every $t\geq{t_0}$, where $|\cdot|$ denotes the Euclidean norm. \hfill$\bullet$
\end{definition}%
If \cref{eq:01} is time- and parameter-independent, then the property of being SGPUAS reduces to the usual property of a system being \emph{globally asymptotically stable} (GAS).

%
%
%
%

\section{A new interpretation of a classical extremum seeking scheme}\label{sec:03}
In this section, we investigate how the extremum seeking scheme in \Cref{figure:01} from reference \cite{Krstic2000} can be used for global optimization in the presence of local extrema.

%
%

\subsection{Problem statement and control law}\label{sec:03:1}
Consider a single-input single-output system of the form%
\begin{equation}\label{eq:03}
\dot{x} \ = \ f(x,u), \qquad y \ = \ h(x),
\end{equation}%
where $x\in\mathbb{R}^n$ is the state, $u\in\mathbb{R}$ is the input, $y\in\mathbb{R}$ is the output, and $f\colon\mathbb{R}^n\times\mathbb{R}\to\mathbb{R}^n$ and $h\colon\mathbb{R}^n\to\mathbb{R}$ are smooth. The goal is to asymptotically stabilize the system around a state where the output function $h$ attains a maximum value. Suppose that a smooth control law $u=\alpha(x,\theta)$ is given, which depends on a scalar parameter $\theta$. It is assumed that the resulting parameter-dependent system%
\begin{equation}\label{eq:04}
\dot{x} \ = \ f(x,\alpha(x,\theta))
\end{equation}%
has the following equilibrium and stability properties.%
\begin{assumption}[{\cite[Assumption~1]{Tan2009}}]\label{assumption:01}
There exists a smooth function $l\colon\mathbb{R}\to\mathbb{R}^n$ such that, for every $x\in\mathbb{R}^n$ and every $\theta\in\mathbb{R}$, the following equivalence holds: $f(x,\alpha(x,\theta))=0$ if and only if $x=l(\theta)$.\hfill$\bullet$
\end{assumption}%
\begin{assumption}[{\cite[Assumption~2]{Tan2009}}]\label{assumption:02}
There exists $\beta\in\mathcal{KL}$ such that, for every fixed control parameter $\theta\in\mathbb{R}$, we have%
\begin{equation}\label{eq:05}
|x(t)-l(\theta)| \ \leq \ \beta(|x(0)-l(\theta)|, t)
\end{equation}%
for every solution $x$ of \cref{eq:04} and every $t\geq0$.\hfill$\bullet$
\end{assumption}%
Under the above assumptions, we will analyze the classical extremum seeking scheme in \Cref{figure:01}, which is described by the equations%
\begin{subequations}\label{eq:06}%
\begin{align}
\dot{x} & \ = \ f\big(x,\alpha\big(x,\hat{\theta}+a\,\sin\epsilon{t}\big)\big), \label{eq:06:a} \allowdisplaybreaks \\
\dot{\hat{\theta}} & \ = \ k\,\xi, \label{eq:06:b} \allowdisplaybreaks \\
\dot{\xi} & \ = \ - \omega_l\,\xi + \omega_l\,(y-\eta)\,a\,\sin(\epsilon{t}), \label{eq:06:c} \allowdisplaybreaks \\
\dot{\eta} & \ = \ - \omega_h\,\eta + \omega_h\,y, \label{eq:06:d}
\end{align}%
\end{subequations}%
where $\epsilon$, $\omega_h$, $\omega_l$, $k$, $a$ are suitably selected positive parameters. We investigate \cref{eq:06} for the choice of parameters%
\begin{equation}\label{eq:07}%
\omega_h \ = \ \epsilon\,\delta\,\omega_H', \qquad \omega_l \ = \ \epsilon\,\delta\,\omega_L', \qquad k \ = \ \epsilon\,\delta\,K',
\end{equation}%
where $\epsilon,\delta$ are sufficiently small positive control parameter and $\omega_H'$, $\omega_L'$, and $K'$ are (not necessarily small) positive constants.%
\begin{remark}\label{remark:01}
Our choice of parameters in \cref{eq:07} is the same as in \cite{Krstic2000} and \cite{Tan2006}. The choice of parameters in \cref{eq:07} is also used for a different control scheme in \cite{Tan2009} for global extremum seeking in the presence of local extrema. We study \cref{eq:06} for the choice of parameters \cref{eq:07} in the time scale
\begin{equation}\label{eq:08}
\tau \ = \ \epsilon\,\delta\,{t}
\end{equation}
in the limits $\epsilon,\delta\to0$. Compared to the analysis in \cite{Krstic2000}, \cite{Tan2006}, \cite{Tan2009}, we provide a novel interpretation of the limit system, namely in terms of a locally averaged objective function. As we shall see below, this novel interpretation of the limit system provides a precise explanation why the scheme in \cref{eq:06} has the ability to overcome local extrema. \hfill$\bullet$
\end{remark}%
\begin{figure}%
\centering\includegraphics[scale=1.05]{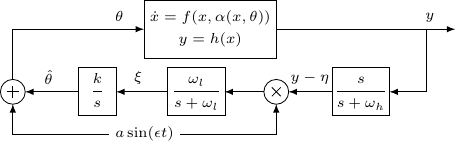}%
\caption{A classical extremum seeking scheme from \cite{Krstic2000}.}%
\label{figure:01}%
\end{figure}%

%
%

\subsection{Analysis and theoretical results}\label{sec:03:2}
To formulate a stability result in the terminology of \Cref{definition:01}, we use the time scale $\tau=\epsilon\delta{t}$ of the filters. Then, for the choice of parameters in \cref{eq:07}, system \cref{eq:06} reads%
\begin{subequations}\label{eq:09}%
\begin{align}
\epsilon\,\delta\,\dot{x} & \ = \ f\big(x,\alpha\big(x,\hat{\theta}+a\,\sin\big(\tfrac{\tau}{\delta}\big)\big)\big), \label{eq:09:a} \allowdisplaybreaks \\
\dot{\hat{\theta}} & \ = \ K'\,\xi, \label{eq:09:b} \allowdisplaybreaks \\
\dot{\xi} & \ = \ - \omega_L'\,\xi + \omega_L'\,\big(h(x)-\eta\big)\,a\,\sin\big(\tfrac{\tau}{\delta}\big), \label{eq:09:c} \allowdisplaybreaks \\
\dot{\eta} & \ = \ - \omega_H'\,\eta + \omega_H'\,h(x). \label{eq:09:d}
\end{align}%
\end{subequations}%
Let $l\colon\mathbb{R}\to\mathbb{R}^n$ be the smooth map from \Cref{assumption:01}. Define the steady-state output function $\psi\colon\mathbb{R}\to\mathbb{R}$ by%
\begin{equation}\label{eq:10}
\psi(\hat{\theta}) \ := \ h(l(\hat{\theta})).
\end{equation}%
The intention of the extremum seeking scheme \cref{eq:09} is maximization of $\psi$. If we ``freeze'' the state $x$ in \cref{eq:09:a} at its equilibrium value $x=l(\hat{\theta}+a\sin(\tau/\delta))$ and substitute it into \cref{eq:09:c} and \cref{eq:09:d}, then we get the so-called ``reduced system''%
\begin{subequations}\label{eq:11}%
\begin{align}
\ddot{\tilde{\theta}} & \ = \ - \omega_L'\,\dot{\tilde{\theta}} \label{eq:11:a} \\
& \qquad + K'\,\omega_L'\,\big(\psi(\tilde{\theta}+a\,\sin\big(\tfrac{\tau}{\delta}\big))-\eta\big)\,a\,\sin\big(\tfrac{\tau}{\delta}\big), \label{eq:11:b} \allowdisplaybreaks \\
\dot{\tilde{\eta}} & \ = \ - \omega_H'\,\tilde{\eta} + \omega_H'\,\psi(\tilde{\theta}+a\,\sin\big(\tfrac{\tau}{\delta}\big)). \label{eq:11:c}
\end{align}%
\end{subequations}%
A standard averaging argument can be applied to relate the ``reduced system'' \cref{eq:11} to its averaged system%
\begin{subequations}\label{eq:12}%
\begin{align}
\ddot{\bar{\theta}} & \ = \ - \omega_L'\,\dot{\bar{\theta}} + K'\,\omega_L'\,\bar{G}_a(\bar{\theta}), \allowdisplaybreaks \\
\dot{\bar{\eta}} & \ = \ - \omega_H'\,\bar{\eta} + \omega_H'\,\bar{z}_a(\bar{\theta}),
\end{align}%
\end{subequations}%
where%
\begin{subequations}\label{eq:13}%
\begin{align}
\bar{G}_a(\bar{\theta}) & \ := \ \frac{1}{2\pi}\int_0^{2\pi}\psi(\bar{\theta}+a\,\sin(\sigma))\,a\,\sin(\sigma)\,\mathrm{d}\sigma, \allowdisplaybreaks \\
\bar{z}_a(\bar{\theta}) & \ := \ \frac{1}{2\pi}\int_0^{2\pi}\psi(\bar{\theta}+a\,\sin(\sigma))\,\mathrm{d}\sigma.
\end{align}%
\end{subequations}%
A direct computation, using integration by parts and integration by substitution, reveals that%
\begin{equation}\label{eq:14}
\bar{G}_a(\bar{\theta}) \ = \ \tfrac{1}{2}\,a^2\,\bar{\psi}_a'(\bar{\theta}),
\end{equation}%
where $\bar{\psi}_a\colon\mathbb{R}\to\mathbb{R}$ is the locally averaged steady-state output function defined by%
\begin{equation}\label{eq:15}
\bar{\psi}_a(\bar{\theta}) \ := \ \frac{1}{2\,a}\int_{-a}^{a}\psi(\bar{\theta}+\vartheta)\,\frac{4}{\pi}\,\sqrt{1-\big(\tfrac{\vartheta}{a}\big)^2}\,\mathrm{d}\vartheta.
\end{equation}%
Consequently, the averaged system \cref{eq:12} can be written as%
\begin{subequations}\label{eq:16}%
\begin{align}
\ddot{\bar{\theta}} & \ = \ - \omega_L'\,\dot{\bar{\theta}} + \tfrac{1}{2}\,K'\,\omega_L'\,a^2\,\bar{\psi}_a'(\bar{\theta}), \label{eq:16:a} \allowdisplaybreaks \\
\dot{\bar{\eta}} & \ = \ - \omega_H'\,\bar{\eta} + \omega_H'\,\bar{z}_a(\bar{\theta}). \label{eq:16:b}
\end{align}%
\end{subequations}%
One may also view $\bar{z}_a$ as a local average of $\psi$, which, however, differs from the local average $\bar{\psi}_a$.%
\begin{remark}\label{remark:02}
In the above averaging procedure, the state~$\tilde{\theta}$ in \cref{eq:11:a}--\cref{eq:11:b} may be interpreted as the position of a double integrator point mass subject to damping and periodic forcing. For sufficiently small $\epsilon>0$, the solutions of \cref{eq:11:a}--\cref{eq:11:b} approximate the solutions of \cref{eq:16:a}. One may interpret \cref{eq:16:a} again as double integrator point mass, which is driven by a ``gradient force'' due to the locally averaged objective function $\bar{\psi}_a$ defined in \cref{eq:15}. Taking the integral of $\psi$ around $\bar{\theta}$ is a convolutional operation, which naturally induces a smoothing of the original function $\psi$, where the smoothing kernel has a compact support and a symmetric shape. In \cref{eq:15}, the average of $\psi$ is taken over an integral of length $2a$ centered at the current state $\bar{\theta}$. Such a local average can ``wash out'' local extrema of $\psi$, which may be helpful for global optimization. For $a\to0$, the function $\bar{\psi}_a$ converges locally uniformly to the original steady-state output function $\psi$. Hence, if the local average in \cref{eq:15} eliminates local extrema of $\psi$, then an approximation of the solutions of \cref{eq:16:a} may be beneficial for convergence into a neighborhood of a global maximizer of the original function $\psi$.\hfill$\bullet$
\end{remark}%
Now we state the main theoretical results of this section.%
\begin{theorem}\label{theorem:01}
Suppose that \Cref{assumption:01,assumption:02} are satisfied. Suppose there exists $\bar{\theta}_a^\ast\in\mathbb{R}$ such that the double integrator \cref{eq:16:a} is GAS w.r.t.~$[\bar{\theta}_a^\ast,0]^\top$. Then the extremum seeking system \cref{eq:09} is SGPUAS w.r.t.~$[l(\bar{\theta}_a^\ast),\bar{\theta}_a^\ast,0,z(\bar{\theta}_a^\ast)]^\top$.\hfill$\bullet$
\end{theorem}%
\begin{proof}
Note that \cref{eq:16} is a cascade of a GAS system followed by an input-to-state stable system. For this reason, the averaged system \cref{eq:16} is GAS w.r.t.~$[\bar{\theta}_a^\ast,0,\bar{z}_a(\bar{\theta}_a^\ast)]^\top$. By \cite[Theorem~1]{Moreau2000}, this implies that the reduced system \cref{eq:11} is SGPUAS w.r.t.~$[\bar{\theta}_a^\ast,0,\bar{z}_a(\bar{\theta}_a^\ast)]^\top$. Now the claim follows from \Cref{assumption:01,assumption:02} as well as \cite[Lemma~1]{Tan2006}.
\end{proof}%
The following assumption provides a sufficient condition for the double integrator \cref{eq:16:a} to be GAS.%
\begin{assumption}\label{assumption:03}
There exists $\bar{\theta}_a^\ast\in\mathbb{R}$ such that, for every $\bar{\theta}\in\mathbb{R}$ with $\bar{\theta}\neq\bar{\theta}_a^\ast$, we have $\bar{\psi}_a'(\bar{\theta})(\bar{\theta}-\bar{\theta}_a^\ast)<0$. \hfill$\bullet$
\end{assumption}%
\begin{proposition}\label{proposition:01}
Suppose that \Cref{assumption:03} is satisfied with $\bar{\theta}_a^\ast\in\mathbb{R}$ as therein. Then the double integrator \cref{eq:16:a} is GAS w.r.t.~$[\bar{\theta}_a^\ast,0]^\top$.\hfill$\bullet$
\end{proposition}%
For the sake of completeness, a proof of \Cref{proposition:01} is provided in the \hyperlink{appendix}{Appendix}.%
\begin{remark}\label{remark:03}
In practice, it is typically impossible to check whether \Cref{assumption:03} is satisfied, since the steady-state output function $\psi$ is analytically unknown. A suitable choice of the parameter $a$ requires knowledge about the maximum width and depth of local extrema of $\psi$. Moreover, a global maximizer $\bar{\theta}^\ast_a$ of the averaged function $\bar{\psi}_a$ is not necessarily a global maximizer $\theta^\ast$ of the original function $\psi$; see \Cref{sec:03:3} for an example. However, if \Cref{assumption:03} is satisfied with $\bar{\theta}_a^\ast$ as therein, then $\bar{\theta}^\ast_a\to\theta^\ast$ as $a\to0$ since $\bar{\psi}^\ast_a$ converges locally uniformly to $\psi$ as $a\to0$.\hfill$\bullet$
\end{remark}%

%
%

\subsection{Simulation example}\label{sec:03:3}
We apply the extremum seeking scheme \cref{eq:06} to the following problem from \cite{Tan2009}. The single-input single-output system \cref{eq:03} is assumed to take the form%
\begin{align}\label{eq:17}
& \dot{x}_1 \ = \ - x_1 + x_2, \qquad \dot{x}_2 \ = \ x_2 + u, \allowdisplaybreaks \\
& y = - (x_1+3x_2)^4 + \tfrac{8}{15}(x_1+3x_2)^3 + \tfrac{6}{5}(x_1+3x_2)^2 + 10.\nonumber
\end{align}%
The stabilizing feedback law $u = - x_1 - 4\,x_2 + \theta$ ensures that \Cref{assumption:01,assumption:02} are satisfied with $l\colon\mathbb{R}\to\mathbb{R}^2$ given by $l(\theta)=\frac{1}{4}\,\theta\,[1,1]^\top$. Hence, the steady-state output function $\psi$ in \cref{eq:10} is given by%
\begin{equation}\label{eq:18}
\psi(\theta) \ = \ - \theta^4 + \tfrac{8}{15}\,\theta^3 + \tfrac{6}{5}\,\theta^2 + 10.
\end{equation}%
The function $\psi$ attains a global maximum at $\theta=+1$ and a local maximum at $\theta=-0.6$. The locally averaged steady-state output function $\bar{\psi}_a$ in \cref{eq:15} is explicitly given by%
\begin{equation}\label{eq:19}
\begin{split}
\bar{\psi}_a(\theta) & \ = \ - \theta^4 + \tfrac{8}{15}\,\theta^3 + \big(\tfrac{6}{5} - \tfrac{3}{2}\,a^2\big)\,\theta^2 \\
& \qquad + \tfrac{2}{5}\,a^2\,\theta + 10 + \tfrac{3}{10}\,a^2 - \tfrac{1}{8}\,a^2.
\end{split}
\end{equation}%
We study the behavior of \cref{eq:06} for the constants $\omega_H'=\omega_L'=K'=1$, the parameters $\delta=1$, $\varepsilon=\frac{1}{100}$, and the initial values $x_1(0)=x_2(0)=0$, $\hat{\theta}(0)=-1$, $\xi(0)=0$, $\eta(0)=y(0)$. Plots of $\bar{\psi}_a$ for $a=0.4$ and $a=0.7$ are shown in \Cref{figure:02}. One can see that the local maximizer of $\psi$ at $\theta=-0.6$ is not fully ``washed out'' in the local average $\bar{\psi}_{0.4}$ but persists at $\theta=-0.5$. It is therefore not surprising that the extremum seeking scheme is stuck at this local maximizer. For $a=0.7$, the local maximizer of $\psi$ is fully eliminated (\Cref{assumption:03} is satisfied) and the extremum seeking scheme converges into a neighborhood of the global maximizer of $\bar{\psi}_{0.7}$ at $\theta=0.8$. This illustrates how a local average of $\psi$ can help to overcome local extrema. On the other hand, a persistently large amplitude $a>0$ is not desirable in practical applications. The choice of the parameter~$a$ represents a trade-off between the desirable feature of overcoming local extrema and the undesirable feature of a large ultimate error due to persistent oscillations. For this reason, it makes sense to implement the method with an initially large and then slowly decreasing amplitude. Such an approach is also proposed in \cite{Tan2009} for a similar extremum seeking scheme, but without the interpretation in terms of a locally averaged steady-state output function. For suitably chosen amplitude dynamics, the extremum seeking scheme converges to the global maximizer of $\psi=\lim_{a\downarrow0}\bar{\psi}_{a}$, as shown in \Cref{figure:03}.%
\begin{figure}%
\centering\includegraphics[scale=1.05]{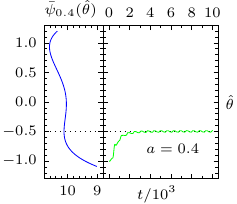}\includegraphics[scale=1.05]{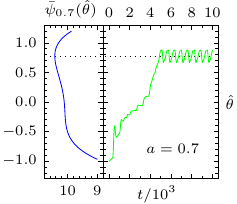}%
\caption{Plot of $\bar{\psi}_a$ (blue) in \cref{eq:19} and a trajectory of the state component $\hat{\theta}$ (green) in \cref{eq:06:b} for $a=0.4$ (left) and $a=0.7$ (right).}%
\label{figure:02}%
\end{figure}%
\begin{figure}%
\centering\includegraphics[scale=1.05]{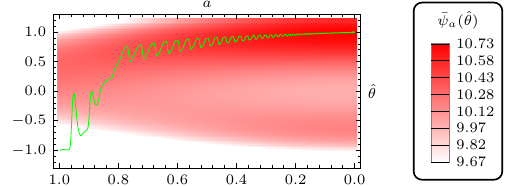}%
\caption{Plot of the averaged steady-state output function $\bar{\psi}_a$ in \cref{eq:19}. The trajectory in green represents the state component $\hat{\theta}$ in \cref{eq:06:b} with constant $a>0$ replaced by the amplitude dynamics $\dot{a}(t)=-\epsilon^2\,a(t)$, $a(0)=1$.}%
\label{figure:03}%
\end{figure}%

%
%
%
%

\section{Global source seeking with a two-dimensional double integrator}\label{sec:04}
In this section, we investigate the problem of global source seeking with a two-dimensional double integrator point mass.

%
%

\subsection{Problem statement and control law}\label{sec:04:1}
Consider a two-dimensional double integrator point%
\begin{equation}\label{eq:20}
m\,\ddot{q} \ = \ -\kappa\,\dot{q} + f
\end{equation}%
of mass $m>0$ with unknown linear damping constant $\kappa>0$, unknown position $q\in\mathbb{R}^2$, and unknown velocity $\dot{q}\in\mathbb{R}^2$. The applied control force is denoted by $f\in\mathbb{R}^2$. Let $\psi\colon\mathbb{R}^2\to\mathbb{R}$ be an analytically unknown smooth function, which describes an unknown signal in $\mathbb{R}^2$. It is assumed that the double integrator point can measure the value of $\psi$ at its current position $q$. The task of the double integrator is to locate a point where the signal attains its global maximum value. This point of maximum strength is called the \emph{source} of the signal. The signal may have other local extrema. 

To overcome local extrema, we now propose a method that provides access to the gradient of a locally averaged objective function $\bar{\psi}_M$, which is given by the local average of the unknown signal function $\psi$ over a user-prescribed submanifold $M$. To this end, let $M$ be a 2-dimensional compact smooth \begin{wrapfigure}{r}{2.7cm}
\includegraphics[scale=1.05]{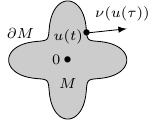}
\end{wrapfigure} submanifold of $\mathbb{R}^2$ with connected boundary $\partial{M}$. Let $\nu$ be the outward-pointing unit normal vector field along $\partial{M}$. Let $T$ be a positive real number. Assume that $u\colon\mathbb{R}\to\partial{M}$ is a smooth, zero-mean, and $T$-periodic curve in $\partial{M}$ with nowhere vanishing derivative such that the restriction of $u$ to the open interval $(0,T)$ is injective. Thus, the restriction of $u$ to $(0,T)$ is a global parametrization of $\partial{M}$ up to the point $u(0)=u(T)$. Let $A(M)$ be the area of $M$ and let $c$ and $\omega_H$ be arbitrary positive constants. For sufficiently small parameter $\epsilon>0$, we propose the control force%
\begin{equation}\label{eq:21}
f \ = \ m\,\tfrac{1}{\epsilon^2}\,\ddot{u}\big(\tfrac{1}{\epsilon}t\big) + \tfrac{c\,T}{A(M)}\,(y-\eta)\,\big|\dot{u}\big(\tfrac{1}{\epsilon}t\big)\big|\,\nu\big(u\big(\tfrac{1}{\epsilon}t\big)\big),
\end{equation}%
where $y=\psi(q)$ is the measured value of $\psi$ at $q$ and $\eta$ is the state of the high-pass filter%
\begin{equation}\label{eq:22}
\dot{\eta} \ = \ - \omega_H\,\eta + \omega_H\,y
\end{equation}%
with filter input $y$ and filter output $y-\eta$.%
\begin{remark}\label{remark:04}
If the mass $m$ of the double integrator point is unknown, then, instead of \cref{eq:21}, one can apply the force%
\begin{equation}\label{eq:23}
f \ = \ \mu\,\tfrac{1}{\epsilon^2}\,\ddot{u}\big(\tfrac{1}{\epsilon}t\big) + \tfrac{c\,T}{A(M)}\,(y-\eta)\,\big|\dot{u}\big(\tfrac{1}{\epsilon}t\big)\big|\,\nu\big(u\big(\tfrac{1}{\epsilon}t\big)\big)
\end{equation}%
with an arbitrary positive constant $\mu$ in place of $m$. In this case, one can simply re-scale $\tilde{u}(\tau):=\frac{\mu}{m}\,u(\tau)$, $\tilde{c}:=\frac{\mu}{m}\,c$, $\tilde{M}:=\frac{\mu}{m}M$, and $\tilde{\nu}(\tilde{p}):=\nu(\frac{m}{\mu}\,\tilde{p})$. Then \cref{eq:23} can be written again of the form \cref{eq:21}; namely%
\begin{equation}\label{eq:24}
f \ = \ m\,\tfrac{1}{\epsilon^2}\,\ddot{\tilde{u}}\big(\tfrac{1}{\epsilon}t\big) + \tfrac{c\,T}{A(\tilde{M})}\,(y-\eta)\,\big|\dot{\tilde{u}}\big(\tfrac{1}{\epsilon}t\big)\big|\,\tilde{\nu}\big(\tilde{u}\big(\tfrac{1}{\epsilon}t\big)\big).
\end{equation}%
Hence, the subsequent considerations also apply to \cref{eq:23} but with the above re-scaling.\hfill$\bullet$
\end{remark}%

%
%

\subsection{Analysis and theoretical results}\label{sec:04:2}
The double integrator point \cref{eq:20} under control law \cref{eq:21}, \cref{eq:22} is described by the equations%
\begin{subequations}\label{eq:25}%
\begin{align}
m\,\ddot{q} & \ = \ - \kappa\,\dot{q} + m\,\tfrac{1}{\epsilon^2}\,\ddot{u}\big(\tfrac{1}{\epsilon}t\big) \label{eq:25:a} \allowdisplaybreaks \\
& \qquad + \tfrac{c\,T}{A(M)}\,(\psi(q)-\eta)\,\big|\dot{u}\big(\tfrac{1}{\epsilon}t\big)\big|\,\nu\big(u\big(\tfrac{1}{\epsilon}t\big)\big), \label{eq:25:b}\allowdisplaybreaks \\
\dot{\eta} & \ = \ - \omega_H\,\eta + \omega_H\,\psi(q). \label{eq:25:c}
\end{align}%
\end{subequations}%
Let $U\colon\mathbb{R}\to\mathbb{R}^2$ be the zero-mean anti-derivative of $u$. We perform the change of variables%
\begin{equation}\label{eq:26}
\tilde{q} \ = \ q - u\big(\tfrac{1}{\epsilon}t\big) + \epsilon\,\tfrac{\kappa}{m}\,U\big(\tfrac{1}{\epsilon}t\big), \qquad \tilde{\eta} \ = \ \eta.
\end{equation}%
In the variables \cref{eq:26}, the closed-loop system \cref{eq:25} reads%
\begin{subequations}\label{eq:27}%
\begin{align}
\!\!m\,\ddot{\tilde{q}} & \ = \ -\kappa\,\dot{\tilde{q}} \label{eq:27:a} \allowdisplaybreaks \\
& \!\!\!\!\!\!\!\!\!\!\! + \tfrac{cT}{A(M)}\psi\big(\tilde{q} + u\big(\tfrac{1}{\epsilon}t\big) - \epsilon\tfrac{\kappa}{m}U\big(\tfrac{1}{\epsilon}t\big)\big)\big|\dot{u}\big(\tfrac{1}{\epsilon}t\big)\big|\nu\big(u\big(\tfrac{1}{\epsilon}t\big)\big) \label{eq:27:b} \allowdisplaybreaks \\
& - \tfrac{c\,T}{A(M)}\,\eta\,\big|\dot{u}\big(\tfrac{1}{\epsilon}t\big)\big|\,\nu\big(u\big(\tfrac{1}{\epsilon}t\big)\big) + \tfrac{\kappa^2}{m}\,u\big(\tfrac{1}{\epsilon}t\big), \label{eq:27:c} \allowdisplaybreaks \\
\dot{\tilde{\eta}} & \ = \ - \omega_H\,\tilde{\eta} + \omega_H\,\psi\big(\tilde{q} + u\big(\tfrac{1}{\epsilon}t\big) - \epsilon\,\tfrac{\kappa}{m}\,U\big(\tfrac{1}{\epsilon}t\big)\big). \label{eq:27:d}
\end{align}%
\end{subequations}%
In the limit $\epsilon\to0$, one can relate \cref{eq:27} to an averaged system. The term in \cref{eq:27:b} provides the averaged contribution%
\begin{equation}\label{eq:28}
\bar{G}_M(\bar{q}) \ := \ \frac{c\,T}{A(M)}\,\frac{1}{T}\int_0^T\psi\big(\bar{q} + u(\tau)\big)\,\big|\dot{u}(\tau)\big|\,\nu(u(\tau))\,\mathrm{d}\tau.
\end{equation}%
Because $u$ is a parametrization of $\partial{M}$, we have%
\begin{equation}\label{eq:29}
\bar{G}_M(\bar{q}) \ = \ \frac{c}{A(M)}\int_{\partial{M}}\psi\big(\bar{q} + p\big)\,\nu(p)\,\mathrm{d}\lambda_{\partial{M}}(p),
\end{equation}%
where $\lambda_{\partial{M}}$ denotes the Lebesgue measure on $M$. Using the divergence theorem, we obtain%
\begin{equation}\label{eq:30}
\bar{G}_M(\bar{q}) \ = \ c\,\nabla\bar{\psi}_M(\bar{q}),
\end{equation}%
where $\bar{\psi}_M\colon\mathbb{R}^2\to\mathbb{R}$ is the locally averaged signal function defined by%
\begin{equation}\label{eq:31}
\bar{\psi}_M(q) \ := \ \frac{1}{A(M)}\int_{M}\psi(q+p)\,\mathrm{d}p.
\end{equation}%
In the limit $\epsilon\to0$, the terms in \cref{eq:27:c} do not contribute to the averaged system while the last term in \cref{eq:27:d} provides%
\begin{equation}\label{eq:32}
\bar{z}_M(\bar{q}) \ := \ \frac{1}{T}\int_0^T\psi(\bar{q} + u(\tau))\,\mathrm{d}\tau.
\end{equation}%
This indicates that the averaged system for \cref{eq:27} reads%
\begin{subequations}\label{eq:33}%
\begin{align}
m\,\ddot{\bar{q}} & \ = \ - \kappa\,\dot{\bar{q}} + c\,\nabla\bar{\psi}_M(\bar{q}), \label{eq:33:a} \allowdisplaybreaks \\
\dot{\bar{\eta}} & \ = \ - \omega_H\,\bar{\eta} + \omega_H\,\bar{z}_M(\bar{q}). \label{eq:33:b}
\end{align}%
\end{subequations}%
\begin{remark}\label{remark:05}
In \cref{eq:31}, the local average of $\psi$ is taken over the manifold $M$, which is enclosed by the dither signal $u$. As already observed in the one-dimensional setting of \Cref{sec:03}, such a local average of $\psi$ can ``wash out'' local extrema, which may be helpful for global optimization. The double integrator point mass in \cref{eq:33:a} is driven into the gradient direction of the locally averaged signal function $\bar{\psi}_M$. It is reasonable to expect that $\bar{\psi}_M$ has fewer local minima than the original signal function $\psi$. An explicit example can be found in \Cref{sec:04:3}.
\end{remark}%
\begin{example}\label{example:01}
Suppose that $M$ is the closed disk $a\bar{\mathbb{D}}$ of radius $a>0$ centered at the origin. Then, the boundary of $a\bar{\mathbb{D}}$ is parametrized by $u=a[\cos, \sin]^\top$ and \cref{eq:21} reads%
\begin{equation}\label{eq:34}
f \ = \ a\,\big( - m\,\tfrac{1}{\epsilon^2} + \tfrac{2\,c}{a^2}\,(y-\eta)\big) \begin{bmatrix} \cos(t/\epsilon) \\ \sin(t/\epsilon) \end{bmatrix}.
\end{equation}%
In this case, the averaged signal function $\bar{\psi}_a\colon\mathbb{R}^2\to\mathbb{R}$ with radius $a>0$ is given by%
\begin{equation}\label{eq:35}
\bar{\psi}_a(q) \ := \ \bar{\psi}_{a\bar{\mathbb{D}}}(q) \ = \ \frac{1}{\pi\,a^2}\int_{a\bar{\mathbb{D}}}\psi(q+p)\,\mathrm{d}p.
\end{equation}%
Note that $\bar{\psi}_a$ converges locally uniformly to the original signal function $\psi$ as $a\to0$. Consequently, the smaller $a$, the closer $\bar{\psi}_a$ is to $\psi$. The control force in \cref{eq:34} is similar to the one in \cite{Zhang20072} for source-seeking vehicles. However, the theoretical analysis in \cite{Zhang20072} is limited to local stability properties for a quadratic signal function and without the interpretation in terms of a locally averaged signal. In \cite{Suttner20241}, \cite{Suttner20243}, \cite{Suttner20261}, source seeking methods for unicycles and single integrators are proposed, which allow an approximation of a kinematic gradient system of the form $\dot{\bar{q}}=c\,\nabla\bar{\psi}_a(\bar{q})$ with the same function $\bar{\psi}_a$ as in \cref{eq:35}. The results in the present paper show that the approach from \cite{Suttner20241}, \cite{Suttner20243}, \cite{Suttner20261} is not limited to a kinematic setting but can be extended to a dynamic (\emph{heavy ball}) gradient system of the form \cref{eq:33:a}. A heavy ball-like method may be beneficial for global optimization, since momentum can help to pass through local extrema. Moreover, the proposed method~\cref{eq:21} is not limited to local averages over disks as in \cite{Suttner20241}, \cite{Suttner20243}, \cite{Suttner20261}, but allows averaging over a user-prescribed manifold $M$.  \hfill$\bullet$
\end{example}%
Now we state the main theoretical results of this section.%
\begin{theorem}\label{theorem:02}
Suppose there exists $\bar{q}_M^\ast\in\mathbb{R}^2$ such that the double integrator \cref{eq:33:a} is GAS w.r.t.~$[\bar{q}_M^\ast,0]^\top$. Then the coordinate-transformed source-seeking system \cref{eq:27} is SGPUAS w.r.t.~$[\bar{q}_M^\ast,0,\bar{z}_M(\bar{q}_M^\ast)]^\top$.\hfill$\bullet$
\end{theorem}%
\begin{proof}
Note that \cref{eq:33} is a cascade of the GAS system \cref{eq:33:a} followed by the input-to-state stable system \cref{eq:33:b}. For this reason, the averaged system \cref{eq:33} is GAS w.r.t.~$[\bar{\theta}_M^\ast,0,\bar{z}_M(\bar{\theta}_M^\ast)]^\top$. In the terminology of \cite[Definition~1]{Teel20001}, it is straight-forward to check that system~\cref{eq:33} is the ``weak average'' of \cref{eq:27}. Using \cite[Theorem~1]{Teel20001}, it follows that the solutions of \cref{eq:27} approximate the solutions of \cref{eq:33} on arbitrary large compact time intervals and arbitrary large compact subsets of the state space. By \cite[Theorem~1]{Moreau2000}, the fact that \cref{eq:33} is GAS w.r.t.~$[\bar{\theta}_M^\ast,0,\bar{z}_M(\bar{\theta}_M^\ast)]^\top$ implies that \cref{eq:27} is SGPUAS w.r.t.~$[\bar{q}_M^\ast,0,\bar{z}_M(\bar{q}_M^\ast)]^\top$.
\end{proof}%
A simple sufficient condition for the double integrator \cref{eq:33:a} to be GAS w.r.t.~$[\bar{q}_M^\ast,0]^\top$ is that $\bar{\psi}_M$ has no other critical point than $\bar{q}_M^\ast$ and that $\bar{\psi}_M(q)\to-\infty$ whenever $|q|\to\infty$. A radially unbounded signal is, however, rather unrealistic in the context of source seeking. It is more likely, that the strength of the signal attains its maximum value at the unknown source and that the signal strength tends to zero with increasing distance from the source; for instance, an exponentially decaying signal with spatial inhomogeneities as in \Cref{figure:04}. The subsequent \Cref{assumption:04} contains more realistic conditions for the purpose of source seeking.%
\begin{figure*}%
\centering$\begin{matrix}\includegraphics[scale=1.05]{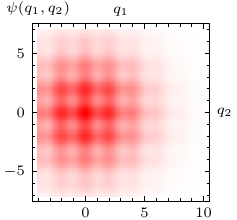}\end{matrix}\qquad\begin{matrix}\includegraphics[scale=1.05]{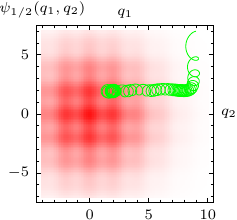}\end{matrix}\qquad\begin{matrix}\includegraphics[scale=1.05]{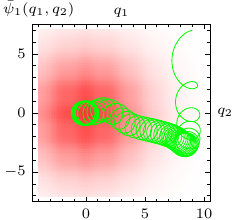}\end{matrix}\qquad\begin{matrix}\includegraphics[scale=1.05]{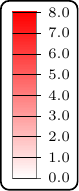}\end{matrix}$%
\caption{Left: Plot of the signal $\psi=\lim_{a\downarrow0}\bar{\psi}_a$ given by \cref{eq:36}. Center and Right: Plot of the averaged signal function $\bar{\psi}_a$ (red) given by \cref{eq:35} and a trajectory of the double integrator point (green) given by \cref{eq:25:b}, \cref{eq:25:c} for $a=1/2$ (center) and $a=1$ (right).}%
\label{figure:04}%
\end{figure*}%
\begin{assumption}\label{assumption:04}
There exists $\bar{q}_M^\ast\in\mathbb{R}^2$ such that%
\begin{enumerate}
	\item $\bar{\psi}_M(\bar{q})<\bar{\psi}_M(\bar{q}_M^\ast)$ for every $\bar{q}\in\mathbb{R}^2$ with $\bar{q}\neq\bar{q}_M^\ast$;
	\item $\nabla\bar{\psi}_M(\bar{q})\neq0$ for every $\bar{q}\in\mathbb{R}^2$ with $\bar{q}\neq\bar{q}_M^\ast$;
	\item $\nabla\bar{\psi}_M(\bar{q})^\top(\bar{q}-\bar{q}_M^\ast)\leq0$ for every $\bar{q}\in\mathbb{R}^2$ with $\bar{q}\notin{C}$, where $C$ is some compact subset of $\mathbb{R}^2$. \hfill$\bullet$
\end{enumerate}%
\end{assumption}%
\begin{proposition}\label{proposition:02}
Suppose that \Cref{assumption:04} is satisfied with $\bar{q}_M^\ast\in\mathbb{R}^2$ as therein. Then the double integrator \cref{eq:33:a} is GAS w.r.t.~$[\bar{q}_M^\ast,0]^\top$.\hfill$\bullet$
\end{proposition}%
For the sake of completeness, a proof of \Cref{proposition:02} is provided in the \hyperlink{appendix}{Appendix}.%

%
%

\subsection{Simulation example}\label{sec:04:3}
As an example of global source seeking in the presence of local extrema, suppose that the unknown signal $\psi\colon\mathbb{R}^2\to\mathbb{R}$ is given by%
\begin{equation}\label{eq:36}
\psi(q) \ = \ \big(6 + \cos(3\,q_1) + \cos(3\,q_2)\big)\,\mathrm{e}^{-(|q|/5)^2}.
\end{equation}%
This signal function is shown in the left plot of \Cref{figure:04}. The global maximizer of $\psi$ is at the origin. But there are also infinitely many other local extrema. We apply the proposed control force \cref{eq:21}, where the dither signal $u$ is chosen as in \Cref{example:01}. Then, the averaged signal function $\bar{\psi}_a$ is given by \cref{eq:35}. The closed-loop system \cref{eq:25} is simulated for the constants $m=\kappa=c=\omega_H=1$, parameter $\epsilon=\frac{1}{10}$, and initial values $q_1(0)=-9$, $q_2(0)=7$, $\eta(0)=0$. Plots of the averaged signal function $\bar{\psi}_a$ for $a=1/2$ and $a=1$ are shown in the center and right plot of \Cref{figure:04}, respectively. One can see in the center plot of \Cref{figure:04} that the double integrator point gets stuck at a local maximum of $\bar{\psi}_{1/2}$, which is not the global maximum at the origin. For $a=1$, the radius of averaging is large enough to eliminate all non-global extrema of $\psi$. One can check (at least numerically) that \Cref{assumption:04} is satisfied for $a=1$. The right plot of \Cref{figure:04} shows that the double integrator tends to a circular motion around the global maximizer at the origin.

In the center and right plot of \Cref{figure:04}, one can observe an initial drift of the double integrator into the direction $a[1,-10]^\top$. This drift can be explained by the change of variables \cref{eq:26}. A zero initial velocity $\dot{q}(0)=[0,0]^\top$ of the point mass in \cref{eq:25} implies the nonzero initial velocity%
\begin{equation}\label{eq:37}
\dot{\tilde{q}}(0) \ = \ \dot{q}(0) - \tfrac{1}{\varepsilon}\,\dot{u}(0) + \tfrac{\kappa}{m}\,u(0) \ = \ a\left[\begin{smallmatrix} 1 \\ -10 \end{smallmatrix}\right]
\end{equation}%
in the transformed system \cref{eq:27}, which approximates \cref{eq:33}.%

%
%
%
%

\section{Conclusions and outlook}\label{sec:05}
We have investigated semi-global stability properties of extremum seeking systems in the presence of local extrema. The basic idea is to get access to a locally averaged objective function. We have seen that such a local average can eliminate local extrema of the original objective function. The area of local averaging is determined by the amplitudes of the employed dither signals. For future research, it might be promising to investigate schemes with adaptive dither amplitudes as in \cite{Ye2020}, \cite{Mimmo2023}. However, it is likely that such an adaptive mechanism will require some \emph{a priori} knowledge about the shape of the unknown objective function.

\setlength{\bibsep}{0pt}
\renewcommand*{\bibfont}{\normalfont\fontsize{8}{10}\selectfont}
\bibliographystyle{unsrt}
\bibliography{bibFile}

\hypertarget{appendix}{}
\appendix
\section*{Proofs of \texorpdfstring{\Cref{proposition:01,proposition:02}}{Proofs of Propositions~\ref{proposition:01} and~\ref{proposition:02}}}
In this appendix, we provide proofs of \Cref{proposition:01,proposition:02}. That is, we show that \cref{eq:16:a} and \cref{eq:33:a} are GAS. Note that each of the double integrators \cref{eq:16:a} and \cref{eq:33:a} can be equivalently written as a system of the form
\begin{equation}\label{eq:38}
\dot{x} \ = \ v, \qquad \dot{v} \ = \ - k\,v - \nabla{V(x)}
\end{equation}
with position $x\in\mathbb{R}^n$ and velocity $v\in\mathbb{R}^n$, where $k$ is a suitably defined positive real number and $V$ is a suitably defined smooth real-valued function on $\mathbb{R}^n$. Throughout this appendix, we assume that there exists $x^\ast\in\mathbb{R}^n$ such that
\begin{equation}\label{eq:39}
V(x) \ > \ V(x^\ast) \ = \ 0 \qquad \text{and} \qquad \nabla{V(x)} \ \neq \ 0
\end{equation}
for every $x\in\mathbb{R}^n$ with $x\neq{x^\ast}$, which corresponds to \Cref{assumption:03} (for $n=1$) and also to the first two items of \Cref{assumption:04} (for $n=2$). In addition, we assume that there exists a compact subset $C$ of $\mathbb{R}^n$ such that $\nabla{V(x)}^\top(x-x^\ast)\geq0$ for every $x\in\mathbb{R}^n$ with $x\notin{C}$. For $n=1$, this follows from \Cref{assumption:03}. For $n=2$, this is part of \Cref{assumption:04}.

In the first step, we show that every maximal solution of \cref{eq:38} remains in a compact set and is therefore forward-complete. To this end, define $B\colon\mathbb{R}^n\times\mathbb{R}^n\to\mathbb{R}$ by
\begin{align}
B(x,v) & \ := \ (1+k^2)\,(V(x)-V(x^\ast)) \label{eq:40} \\
& \qquad + \frac{1}{2}\begin{bmatrix} x-x^\ast \\ v \end{bmatrix}^\top\begin{bmatrix} k^2\,I & k\,I \\ k\,I & (1+k^2)\,I \end{bmatrix}\begin{bmatrix} x-x^\ast \\ v \end{bmatrix}. \nonumber
\end{align}
The function $B$ is continuously differentiable, positive definite, and radially unbounded. Hence, its sublevel sets are compact. Moreover, a direct computation shows that the derivative $\dot{B}\colon\mathbb{R}^n\times\mathbb{R}^n\to\mathbb{R}$ of $B$ along solutions of \cref{eq:38} is given by
\begin{equation}\label{eq:41}
\dot{B}(x,v) \ = \ -k^3\,\|v\|^2 - k\,\nabla{V(x)}^\top(x-x^\ast).
\end{equation}
Because of the assumption that there exists a compact subset $C$ of $\mathbb{R}^n$ such that $\nabla{V(x)}^\top(x-x^\ast)\geq0$ for every $x\in\mathbb{R}^n$ with $x\notin{C}$, we may conclude that every maximal solution of \cref{eq:38} remains in a compact set and is forward-complete.

Define the ``total energy'' function $E\colon\mathbb{R}^n\times\mathbb{R}^n\to\mathbb{R}$ of \cref{eq:38} by
\begin{equation}\label{eq:42}
E(x,v) \ := \ \tfrac{1}{2}\,|v|^2 + V(x).
\end{equation}
Because $V(x)>V(x^\ast)=0$ for every $x\in\mathbb{R}^n$ with $x\neq{x^\ast}$, the function $E$ is positive definite w.r.t.~$[x^\ast,0]^\top$. A direct computation shows that the derivative $\dot{E}\colon\mathbb{R}^n\times\mathbb{R}^n\to\mathbb{R}$ of $E$ along solutions of \cref{eq:38} is given by
\begin{equation}\label{eq:43}
\dot{E}(x,v) \ = \ - k\,|v|^2.
\end{equation}
By Lyapunov's direct method, it follows that \cref{eq:38} is stable w.r.t.~$[x^\ast,0]^\top$.

For the proof that \cref{eq:38} is GAS w.r.t.~$[x^\ast,0]^\top$, it remains to be shown that \cref{eq:38} is globally asymptotically attracted by~$[x^\ast,0]^\top$. To this end, let $[x,v]^\top\colon[\vphantom{]}0,\infty\vphantom{(})\to\mathbb{R}^n\times\mathbb{R}^n$ be a maximal solution of \cref{eq:38}. We already know that $[x,v]^\top$ remains in a compact set. By LaSalle's invariance principle, $[x(t),v(t)]^\top$ converges to the smallest invariant subset of
\begin{equation}\label{eq:44}
\big\{(\tilde{x},\tilde{v})\in\mathbb{R}^n\times\mathbb{R}^n \ \big| \ \dot{E}(\tilde{x},\tilde{v})=0\big\} \ = \ \mathbb{R}^n\times\{0\}
\end{equation}
as $t\to\infty$, which is here the singleton $\{(0,0)\}$. This proves the claim.
\end{document}